\documentclass{amsart}
\usepackage{amsmath,amssymb,amsthm}

\newtheorem{theorem}{Theorem}[section]
\newtheorem{lemma}[theorem]{Lemma}
\newtheorem{proposition}[theorem]{Proposition}
\newtheorem{corollary}[theorem]{Corollary}

\newtheorem{remark}[theorem]{Remark}
\newtheorem{example}[theorem]{Example}

\begin{document}
\title[Comparison of some purities]{Comparison of some purities, flatnesses and injectivities}
\author{Walid Al-Kawarit}
\address{Universit\'e de Caen Basse-Normandie, CNRS UMR
  6139 LMNO,
F-14032 Caen, France}
\email{walid.al-kawarit@unicaen.fr} 

\author{Fran\c cois Couchot}
\address{Universit\'e de Caen Basse-Normandie, CNRS UMR
  6139 LMNO,
F-14032 Caen, France}
\email{francois.couchot@unicaen.fr}

\keywords{$(n,m)$-pure submodule, $(n,m)$-flat module, $(n,m)$-injective module, $(n,m)$-coherent ring}

\subjclass[2000]{Primary 16D40, 16D50, 16D80}

\begin{abstract} In this paper, we  compare $(n,m)$-purities for different pairs of positive integers $(n,m)$. When $R$ is a commutative ring, these purities are not equivalent if $R$ doesn't satisfy the following property: there exists a positive integer $p$ such that, for each maximal ideal $P$, every finitely generated ideal of $R_P$ is $p$-generated.  When this property holds, then the $(n,m)$-purity and the $(n,m')$-purity are equivalent if $m$ and $m'$ are integers $\geq np$. These results are obtained by a generalization of Warfield's methods. There are also some interesting results when $R$ is a semiperfect strongly $\pi$-regular ring. We also compare $(n,m)$-flatnesses and $(n,m)$-injectivities for different pairs of positive integers $(n,m)$. In particular, if $R$ is right perfect and right self $(\aleph_0,1)$-injective, then each $(1,1)$-flat right $R$-module is projective. In several cases, for each positive integer $p$, all $(n,p)$-flatnesses are equivalent. But there are some examples where the $(1,p)$-flatness is not equivalent to the $(1,p+1)$-flatness.
\end{abstract}
\maketitle

All rings in this paper are associative with unity, and all modules are
unital. Let $n$ and $m$ be two positive integers. A right $R$-module $M$ is said to be $(n,m)$-presented if it is the factor module of a free right module of rank $n$ modulo a $m$-generated submodule. A short exact sequence $(\Sigma)$ of left $R$-modules is called $(n,m)$-pure if it remains exact when tensoring it with any $(n,m)$-presented right module.  We say that $(\Sigma)$ is $(\aleph_0,m)$-{\it pure exact} (respectively $(n,\aleph_0)$-{\it pure exact} if, for each positive integer $n$ (respectively $m$) $(\Sigma)$ is $(n,m)$-pure exact. Let us observe that the $(1,1)$-pure exact sequences are exactly the RD-exact sequences (see \cite{War69}) and the $(\aleph_0,\aleph_0)$-exact sequences are the pure-exact sequences in the Cohn's sense. Similar results as in the classical theories of purity hold, with similar proofs. In particular, a left $R$-module is $(n,m)$-pure-projective if and only if it is a summand of a direct sum of $(m,n)$-presented left modules, and each left $R$-module has a $(n,m)$-pure-injective hull which is unique up to an isomorphism. 

In this paper, we  compare $(n,m)$-purities for different pairs of positive integers $(n,m)$. When $R$ is commutative, we shall see that some of these purities are  equivalent only if $R$  satisfies the following property: there exists a positive integer $p$ such that, for each maximal ideal $P$, every finitely generated ideal of $R_P$ is $p$-generated. When this property holds, then  $(n,m)$-purity and  $(n,m')$-purity are equivalent if $m$ and $m'$ are integers $\geq np$. These results are obtained by using the following: if $R$ is a local commutative ring for which there exists a $(p+1)$-generated ideal, where $p$ is a positive integer, then, for each positive integer $n$,  for each integer $m$, $n(p-1)+1\leq m\leq np+1$, there exists a $(n,m)$-presented $R$-module whose endomorphism ring is local. It is a generalization of the Warfield's construction of indecomposable finitely presented modules when $R$ is not a valuation ring. 

When $R$ is semiperfect and strongly $\pi$-regular, we show that there exists an integer $m>0$ such that, for any integer $n>0$, each $(n,m)$-pure exact sequence of right modules is $(n,\aleph_0)$-pure exact if and only if there exists an integer $p>0$ such that every finitely generated left ideal is $p$-generated.

As in \cite{ZhChZh05} we define $(n,m)$-flat modules and $(n,m)$-injective modules. We also compare $(n,m)$-flatnesses and $(n,m)$-injectivities for different pairs of positive integers $(n,m)$. In particular, if $R$ is a right perfect ring which is right self $(\aleph_0,1)$-injective, then each $(1,1)$-flat right $R$-module is projective. For many classes of rings, for each positive integer $p$, we show that the $(1,p)$-flatness implies the $(\aleph_0,p)$-flatness, but we have no general result. If $R$ is a local commutative ring with a non finitely generated maximal ideal $P$ satisfying $P^2=0$, then for each positive integer $p$, there exists an $R$-module which is $(\aleph_0,p)$-flat (resp, $(\aleph_0,p)$-injective)  and which is not $(1,p+1)$-flat (resp, $(1,p+1)$-injective). 

As in \cite{ZhChZh05} we define left $(n,m)$-coherent rings. When $R$ is a commutative locally perfect ring which is $(1,1)$-coherent and self $(1,1)$-injective, we show that $R$ is an IF-ring, each $(1,1)$-flat $R$-module is flat and each $(1,1)$-injective $R$-module is FP-injective. For other classes of rings, for each positive integer $p$, we show that the left $(1,p)$-coherence implies the left $(\aleph_0,p)$-coherence, but we have no general result. If $R=V[[X]]$, the power series ring over  a valuation domain $V$ whose order group is not isomorphic to $\mathbb{R}$, then $R$ is a $(\aleph_0,1)$-coherent ring which is not $(1,2)$-coherent.

\section{$(n,m)$-pure exact sequences}
\label{S:pure} 

By using a standard technique, (see for instance \cite[Chapter I, Section 8]{FuSa01}), we can prove the following theorem, and similar results hold if we replace $n$ or $m$ with $\aleph_0$.
\begin{theorem}
\label{T:pure} Assume that $R$ is an algebra over a commutative ring $S$ and $E$ is an injective $S$-cogenerator. Then, for each exact sequence $(\Sigma)$ of left $R$-modules $0\rightarrow A\rightarrow B\rightarrow C\rightarrow 0$, the following conditions are equivalent:
\begin{enumerate}
\item $(\Sigma)$ is $(n,m)$-pure;
\item  for each $(m,n)$-presented left module $G$ the sequence $\mathrm{Hom}_R(G,(\Sigma))$ is exact;
\item every system of $n$ equations over $A$
\[\sum_{j=1}^{m}r_{i,j}x_j=a_i\in A\qquad (i=1,\dots,n)\]
with coefficients $r_{i,j}\in R$ and unknowns $x_1,\dots,x_m$ has a solution in $A$ whenever it is solvable in $B$;
\item the exact sequence of right $R$-modules $\mathrm{Hom}_S((\Sigma),E)$ is $(m,n)$-pure.
\end{enumerate}
\end{theorem}

 Propositions~\ref{P:PureProj} and \ref{P:PureInj} can be deduced from \cite[Theorem 1]{Warf69}.

A left $R$-module $G$ is called $(n,m)$-{\it pure-projective} if for each $(n,m)$-pure  exact sequence $0\rightarrow A\rightarrow B\rightarrow C\rightarrow 0$ the sequence
\[0\rightarrow \mathrm{Hom}_R(G,A)\rightarrow \mathrm{Hom}_R(G,B)\rightarrow \mathrm{Hom}_R(G,C)\rightarrow 0\]
is exact. Similar definitions can be given by replacing $n$ or $m$ by $\aleph_0$. From Theorem~\ref{T:pure} and by using standard technique (see for instance \cite[Chapter VI, Section 12]{FuSa01}) we get the following proposition in which $n$ or $m$ can be replaced by $\aleph_0$:
\begin{proposition}
\label{P:PureProj} Let $G$ be a left $R$-module. Then the following assertions hold:
\begin{enumerate}
\item there exists a $(n,m)$-pure exact sequence of left modules \[0\rightarrow K\rightarrow F\rightarrow G\rightarrow 0\] where $F$ is a direct sum of $(m,n)$-presented left modules;
\item $G$ is $(n,m)$-pure projective if and only if it is a summand of a direct sum of $(m,n)$-presented left modules.
\end{enumerate}
\end{proposition}
A left $R$-module $G$ is called $(n,m)$-{\it pure-injective} if for each $(n,m)$-pure  exact sequence $0\rightarrow A\rightarrow B\rightarrow C\rightarrow 0$ the sequence
\[0\rightarrow \mathrm{Hom}_R(C,G)\rightarrow \mathrm{Hom}_R(B,G)\rightarrow \mathrm{Hom}_R(A,G)\rightarrow 0\]
is exact. 

If $M$ is a left module we put $M^{\sharp}=\mathrm{Hom}_{\mathbb{Z}}(M,\mathbb{Q}/\mathbb{Z})$. Thus $M^{\sharp}$ is a right module. It is the {\it character module} of $M$.

If $A$ is a submodule of a left $R$-module $B$, we say that $B$ is a $(n,m)$-{\it pure essential extension} of $A$ if $A$ is a $(n,m)$-pure submodule of $B$ and for each nonzero submodule $K$ of $B$ such that $A\cap K=0$, $(A+K)/K$ is not a $(n,m)$-pure submodule of $B/K$. If, in addition, $B$ is $(n,m)$-pure injective, we say that $B$ is a $(n,m)$-{\it pure injective hull} of $A$. In these above definitions and in the following proposition $n$ or $m$ can be  replaced by $\aleph_0$.

\begin{proposition}
\label{P:PureInj} The following assertions hold:
\begin{enumerate}
\item each left $R$-module is a $(n,m)$-pure submodule of a $(n,m)$-pure injective left module;
\item each left $R$-module has a $(n,m)$-pure injective hull which is unique up to an isomorphism.
\end{enumerate}
\end{proposition}
\begin{proof}
(1). Let $M$ be a left $R$-module. By Proposition~\ref{P:PureProj} there exists a $(m,n)$-pure exact sequence of right $R$-modules $0\rightarrow K\rightarrow F\rightarrow M^{\sharp}\rightarrow 0$ where $F$ is a direct sum of $(n,m)$-presented right modules. From Theorem~\ref{T:pure} it  follows that $(M^{\sharp})^{\sharp}$ is a $(n,m)$-pure submodule of $F^{\sharp}$. By \cite[Corollary 1.30]{Fac98} $M$ is isomorphic to a pure submodule of $(M^{\sharp})^{\sharp}$. So, $M$ is isomorphic to a $(n,m)$-pure submodule of $F^{\sharp}$. By using the canonical isomorphism $(F\otimes_R-)^{\sharp}\cong\mathrm{Hom}_R(-,F^{\sharp})$ we get that $F^{\sharp}$ is $(n,m)$-pure injective since $F$ is a direct sum of $(n,m)$-presented modules.

(2). Since (1) holds and every direct limit of $(n,m)$-pure exact sequences is $(n,m)$-pure exact too, we can adapt the method of Warfield's proof of existence of pure-injective hull to show (2)(see \cite[Proposition 6]{War69}). We can also use \cite[Proposition 4.5]{Ste67}.
\end{proof}

\begin{proposition}
\label{P:PureLocal} Let $R$ be a commutative ring and let $(\Sigma)$ be a short exact sequence of $R$-modules. Then $(\Sigma)$ is $(n,m)$-pure if and only if, for each maximal ideal $P$ $(\Sigma)_P$ is $(n,m)$-pure.
\end{proposition}
\begin{proof}
Assume that $(\Sigma)$ is $(n,m)$-pure and let $M$ be a $(n,m)$-presented $R_P$-module where $P$ is a maximal ideal. There exists a $(n,m)$-presented $R$-module $M'$ such that $M\cong M'_P$ and $M\otimes_{R_P}(\Sigma)_P\cong (M'\otimes_R(\Sigma))_P$. We deduce that $(\Sigma)_P$ is $(n,m)$-pure.

Conversely, suppose that $(\Sigma)$ is the sequence $0\rightarrow A\rightarrow B\rightarrow C\rightarrow 0.$ Let $M$ be a $(n,m)$-presented $R$-module. Then, for each maximal ideal $P$, $(\Sigma)_P$ is $(n,m)$-pure over $R$ since $M\otimes_R(\Sigma)_P\cong M_P\otimes_{R_P}(\Sigma)_P$. On the other hand, since $M\otimes_R(\prod_{P\in\mathrm{Max}\ R}(\Sigma)_P)\cong (\prod_{P\in\mathrm{Max}\ R} M\otimes_R(\Sigma)_P)$, $\prod_{P\in\mathrm{Max}\ R}A_P$ is a $(n,m)$-pure submodule of $\prod_{P\in\mathrm{Max}\ R}B_P$. By \cite[Lemme 1.3]{Cou82} $A$ is isomorphic to a pure submodule of $\prod_{P\in\mathrm{Max}\ R}A_P$. We successively deduce that $A$ is a $(n,m)$-pure submodule of $\prod_{P\in\mathrm{Max}\ R}B_P$ and $B$.
\end{proof}

\section{Comparison of purities over a semiperfect ring}
\label{S:semiperfect}

In this section we shall compare $(n,m)$-purities for different pairs of integers $(n,m)$. In \cite{PPR99} some various purities are also compared. In particular some necessary and sufficient conditions on a ring $R$ are given for the $(1,1)$-purity to be equivalent to the $(\aleph_0,\aleph_0)$-purity.

The following lemma is due to Lawrence Levy,  see \cite[Lemma 1.3]{WiWi75}. If $M$ be a finitely generated left (or right) $R$-module, we denote by $\mathrm{gen}\ M$ its minimal number of generators.

\begin{lemma}
\label{L:Levy} Let $R$ be a ring. Assume there exists a positive integer $p$ such that $\mathrm{gen}\ A\leq p$ for each finitely generated left ideal $A$ of $R$. Then $\mathrm{gen}\ N\leq p\times\mathrm{gen}\ M$, if $N$ is a finitely generated submodule of a finitely generated left $R$-module $M$.
\end{lemma}

From this lemma and Theorem~\ref{T:pure} we deduce the following:
\begin{proposition}
\label{P:genIdeal} Let $R$ be a ring. Assume there exists a positive integer $p$ such that $\mathrm{gen}\ A\leq p$ for each finitely generated left ideal $A$ of $R$. Then, for each positive integer $n$:
\begin{enumerate}
\item each $(n,np)$-pure exact sequence of right modules is $(n,\aleph_0)$-pure exact;
\item each $(np,n)$-pure exact sequence of left modules is $(\aleph_0,n)$-pure exact.
\end{enumerate}
\end{proposition}

\begin{corollary}
\label{C:Artinian} Let $R$ be a left Artinian ring. Then there exists a positive integer $p$ such that, for each positive integer $n$:
\begin{enumerate}
\item each $(n,np)$-pure exact sequence of right modules is $(n,\aleph_0)$-pure exact;
\item each $(np,n)$-pure exact sequence of left modules is $(\aleph_0,n)$-pure exact.
\end{enumerate}
\end{corollary}
\begin{proof}
Each finitely generated left $R$-module $M$ has a finite length  denoted by $\mathrm{length}\ M$, and $\mathrm{gen}\ M\leq\mathrm{length}\ M$. So, for each left ideal $A$ we have $\mathrm{gen}\ A\leq\mathrm{length}\ R$.
We choose $p=\sup\{\mathrm{gen}\ A\mid A\ \mathrm{left\ ideal\ of}\ R\}$  and we apply the previous proposition.
\end{proof}

Let $R$ be a ring and $J$ its Jacobson radical. Recall that $R$ is {\it semiperfect} if $R/J$ is semisimple and idempotents lift modulo $J$.

\begin{theorem}
\label{T:semiperfect} Let $R$ be semiperfect ring. Assume that each indecomposable finitely presented cyclic left $R$-module has a local endomorphism ring. The following conditions are equivalent:
\begin{enumerate}
\item there exists an  integer $p>0$ such that, for each  integer $n>0$, each $(n,np)$-pure exact sequence of right modules is $(n,\aleph_0)$-pure exact;
\item there exists an  integer $p>0$ such that, for each  integer $n>0$, each $(np,n)$-pure exact sequence of left modules is $(\aleph_0,n)$-pure exact;
\item there exists an  integer $q>0$ such that each $(1,q)$-pure exact sequence of right modules is $(1,\aleph_0)$-pure exact;
\item there exists an  integer $q>0$ such that each $(q,1)$-pure exact sequence of left modules is $(\aleph_0,1)$-pure exact;
\item there exists an integer $q>0$ such that each indecomposable finitely presented cyclic left module is $q$-related;
\item there exists an integer $p>0$ such that $\mathrm{gen}\ A\leq p$ for each finitely generated left ideal $A$ of $R$.
\end{enumerate}
Moreover, if each indecomposable finitely presented  left $R$-module has a local endomorphism ring, these conditions are equivalent to the following:
\begin{itemize}
\item[(7)] there exist two positive integers $n,\ m$ such that each $(n,m)$-pure exact sequence of right modules is $(n,\aleph_0)$-pure exact;
\item [(8)] there exist two positive integers $n,\ m$ such that  each $(m,n)$-pure exact sequence of left modules is $(\aleph_0,n)$-pure exact;
\end{itemize}
\end{theorem}
\begin{proof}
By Proposition~\ref{P:genIdeal} $(6)\Rightarrow (1)$. By Theorem~\ref{T:pure} $(1)\Leftrightarrow (2)$, $(3)\Leftrightarrow (4)$ and $(7)\Leftrightarrow (8)$. It is obvious that $(2)\Rightarrow (4)$ and $(2)\Rightarrow (7)$.

$(4)\Rightarrow (5)$. Let $C$ be an indecomposable finitely presented cyclic left module. Then $C$ is $(q,1)$-pure-projective. So, $C$ is a direct summand of a finite direct sum of $(1,q)$-presented left modules. Since $R$ is semiperfect, we may assume that these $(1,q)$-presented left modules are indecomposable.  So, by Krull-Schmidt theorem $C$ is $(1,q)$-presented.

$(5)\Rightarrow (6)$. Let $A$ be a finitely generated left ideal. Then $R/A=\oplus_{i=1}^tR/A_i$ where, for each $i=1,\dots,t$, $A_i$ is a left ideal and $R/A_i$ is indecomposable. We have the following commutative diagram with exact horizontal sequences:
\[\begin{matrix}
0 & \rightarrow & \oplus_{i=1}^tA_i & \rightarrow & R^t & \rightarrow & \oplus_{i=1}^tR/A_i & \rightarrow & 0 \\
{} & {} & \downarrow & {} & \downarrow & {} & \downarrow & {} & {} \\
0 & \rightarrow & A & \rightarrow & R & \rightarrow & R/A & \rightarrow & 0 \\
\end{matrix}\]
Since the right vertical map is an isomorphism, we deduce from the snake lemma that the other two vertical homomorphisms have isomorphic cokernels. It follows that $\mathrm{gen}\ A\leq tq+1$ because $\mathrm{gen}\ A_i\leq q$ by $(5)$. On the other hand, let $P$ be a projective cover of $R/A$. Then $P$ is isomorphic to a direct summand of $R$. We know that the left module $R$ is a finite direct sum of indecomposable projective modules. Let $s$ the number of these indecomposable summands. It is easy to show that $t\leq s$. So, if $p=sq+1$, then $\mathrm{gen}\ A\leq p$.

$(8)\Rightarrow (5)$. Let $C$ be an indecomposable finitely presented cyclic left module. Then $C$ is $(m,n)$-pure-projective. So, $C$ is a direct summand of a finite direct sum of $(n,m)$-presented left modules. Since $R$ is semiperfect, we may assume that these $(n,m)$-presented left modules are indecomposable.  So, by the Krull-Schmidt theorem $C$ is $(1,m)$-presented.
\end{proof}

A ring $R$ is said to be \textit{strongly $\pi$-regular} if, for each $r\in R$, there exist $s\in R$ and an integer $q\geq 1$ such that $r^q=r^{q+1}s$. By \cite[Theorem 3.16]{Fac98} each strongly $\pi$-regular $R$ satisfies the following condition: for each $r\in R$, there exist $s\in R$ and an integer $q\geq 1$ such that $r^q=sr^{q+1}$. Recall that a left  $R$-module $M$ is said to be {\it Fitting} if for each endomorphism $f$ of $M$ there exists a positive integer $t$ such that $M=\ker\ f^t\oplus f^t(M)$.

\begin{lemma}
\label{L:Fitting} Let $R$ be a strongly $\pi$-regular semiperfect ring. Then:
\begin{enumerate}
\item each finitely presented cyclic left (or right) $R$-module is Fitting;
\item each indecomposable finitely presented cyclic left (or right) $R$-module has a local endomorphism ring.
\end{enumerate}
\end{lemma}
\begin{proof} In \cite[Lemma 3.21]{Fac98} it is proven that every finitely presented $R$-module is a Fitting module if $R$ is a semiperfect ring with $\mathrm{M}_n(R)$ strongly $\pi$-regular for all $n$. We do a similar proof to show $(1)$.

$(2)$. By \cite[Lemma 2.21]{Fac98} each indecomposable Fitting module has a local endomorphism ring.
\end{proof}

\begin{corollary}
\label{C:semiperfect} Let $R$ be a strongly $\pi$-regular semiperfect ring. The following conditions are equivalent:
\begin{enumerate}
\item there exists an  integer $p>0$ such that, for each  integer $n>0$, each $(n,np)$-pure exact sequence of right modules is $(n,\aleph_0)$-pure exact;
\item there exists an  integer $p>0$ such that, for each  integer $n>0$, each $(np,n)$-pure exact sequence of left modules is $(\aleph_0,n)$-pure exact;
\item there exists an  integer $q>0$ such that each $(1,q)$-pure exact sequence of right modules is $(1,\aleph_0)$-pure exact;
\item there exists an  integer $q>0$ such that each $(q,1)$-pure exact sequence of left modules is $(\aleph_0,1)$-pure exact;
\item there exists an integer $q>0$ such that each indecomposable finitely presented cyclic left module is $q$-related;
\item there exists an integer $p>0$ such that $\mathrm{gen}\ A\leq p$ for each finitely generated left ideal $A$ of $R$.
\end{enumerate}
Moreover, if $\mathrm{M}_n(R)$ is strongly $\pi$-regular for all $n>0$, these conditions are equivalent to the following:
\begin{itemize}
\item[(7)] there exist two positive integers $n,\ m$ such that each $(n,m)$-pure exact sequence of right modules is $(n,\aleph_0)$-pure exact;
\item [(8)] there exist two positive integers $n,\ m$ such that  each $(m,n)$-pure exact sequence of left modules is $(\aleph_0,n)$-pure exact;
\end{itemize}
\end{corollary}
\begin{proof}
 By Lemma~\ref{L:Fitting} each indecomposable finitely presented cyclic left $R$-module has a local endomorphism ring.
If  $\mathrm{M}_n(R)$ is strongly $\pi$-regular for all $n$, then by \cite[Lemmas 3.21 and 2.21]{Fac98} each indecomposable finitely presented left $R$-module has a local endomorphism ring. So, we apply Theorem~\ref{T:semiperfect}.
\end{proof}

Recall that a ring $R$ is {\it right perfect} if each flat right $R$- module is projective.

\begin{corollary}
\label{C:perfect} Let $R$ be a right perfect ring. The following conditions are equivalent:
\begin{enumerate}
\item there exists an  integer $p>0$ such that, for each  integer $n>0$, each $(n,np)$-pure exact sequence of right modules is $(n,\aleph_0)$-pure exact;
\item there exists an  integer $p>0$ such that, for each  integer $n>0$, each $(np,n)$-pure exact sequence of left modules is $(\aleph_0,n)$-pure exact;
\item there exists an  integer $q>0$ such that each $(1,q)$-pure exact sequence of right modules is $(1,\aleph_0)$-pure exact;
\item there exist two positive integers $n,\ m$ such that each $(n,m)$-pure exact sequence of right modules is $(n,\aleph_0)$-pure exact;
\item there exist two positive integers $n,\ m$ such that  each $(m,n)$-pure exact sequence of left modules is $(\aleph_0,n)$-pure exact;
\item there exists an integer $p>0$ such that $\mathrm{gen}\ A\leq p$ for each finitely generated left ideal $A$ of $R$.
\end{enumerate}
\end{corollary}
\begin{proof}
For all $n>0$,  $\mathrm{M}_n(R)$ is right perfect. Since each right perfect ring satisfies the descending chain condition on finitely generated left ideals, then $\mathrm{M}_n(R)$ is strongly $\pi$-regular  for all $n>0$. We apply Corollary~\ref{C:semiperfect}.
\end{proof}

\section{Comparison of purities over a commutative ring}
\label{S:commutative}

In the sequel of this section $R$ is a  commutative local ring, except in Theorem~\ref{T:ComPur}. We  denote respectively by $P$ and $k$ its maximal ideal and  its residue field 

Let $M$ be a finitely presented $R$-module. Recall that $\mathrm{gen}\ M=\dim_k\ M/PM$. Let $F_0$ be a free $R$-module whose rank is $\mathrm{gen}\ M$ and let $\phi:F_0\rightarrow M$ be an epimorphism. Then $\ker\ \phi\subseteq PF_0$. We put $\mathrm{rel}\ M=\mathrm{gen}\  \ker\ \phi$. Let $F_1$ be a free  $R$-module whose rank is $\mathrm{rel}\ M$ and let $f:F_1\rightarrow F_0$ be a homomorphism such that $\mathrm{im}\ f=\ker\ \phi$. Then $\ker\ f\subseteq PF_1$. For any $R$-module $N$, we put $N^*=\mathrm{Hom}_R(N,R)$. Let $f^*:F_0^*\rightarrow F_1^*$ be the homomorphism deduced from $f$. We set $\mathrm{D}(M)=\mathrm{coker}\ f^*$ the Auslander and Bridger's dual of $M$. The following proposition is the version in commutative case of \cite[Theorem 2.4]{War75}:
\begin{proposition}
\label{P:AusBrid} Assume that $M$ has no projective summand. Then:
\begin{enumerate}
\item $\ker\ f^*\subseteq PF_0^*$ and $\mathrm{im}\ f^*\subseteq PF_1^*$;
\item $M\cong \mathrm{D}(\mathrm{D}(M))$ and $\mathrm{D}(M)$ has no projective summand;
\item $\mathrm{gen}\ \mathrm{D}(M)=\mathrm{rel}\ M$ and $\mathrm{rel}\ \mathrm{D}(M)=\mathrm{gen}\ M$;
\item if $M=M_1\oplus M_2$ then\\ $\mathrm{gen}\ M=\mathrm{gen}\ M_1 + \mathrm{gen}\ M_2$ and $\mathrm{rel}\ M=\mathrm{rel}\ M_1 + \mathrm{rel}\ M_2$.
\item $\mathrm{End}_R(\mathrm{D}(M))$ is local if and only if so is $\mathrm{End}_R(M)$.
\end{enumerate}
\end{proposition}
\begin{lemma}
\label{L:iso} Let  $M$ be a finitely generated $R$-module, $s$ an endomorphism of $M$ and $\bar{s}$ the endomorphism of $M/PM$ induced by $s$. Then $s$ is an isomorphism if and only if so is $\bar{s}$.
\end{lemma}
\begin{proof}
If $s$ is an isomorphism it is obvious that so is $\bar{s}$. Conversely, $\mathrm{coker}\ s=0$ by Nakayama lemma. So, $s$ is surjective. By using a Vasconcelos's result (see \cite[Theorem V.2.3]{FuSa01}) $s$ is bijective.
\end{proof}

\begin{proposition}
\label{P:locEnd} Assume that there exists an ideal $A$ with $\mathrm{gen}\ A=p+1$ where $p$ is a positive integer. Then, for each positive integers $n$ and $m$ with $(n-1)p+1\leq m\leq np+1$, there exists a finitely presented $R$-module $W_{p,n,m}$ whose endomorphism ring is local and such that $\mathrm{gen}\ W_{p,n,m}=n$ and $\mathrm{rel}\ W_{p,n,m}=m$. 
\end{proposition}
\begin{proof}
Suppose that $A$ is generated by $a_1,\dots,a_p,a_{p+1}$. Let $F$ be a free module of rank $n$ with basis $e_1,\dots,e_n$ and let $K$ be the submodule of $F$ generated by $x_1,\dots,x_m$ where these elements are defined in the following way: if $j= pq+r$ where $1\leq r\leq  p$, $x_j=a_re_{q+1}$ if $r\ne 1$ or $q=0$, and $x_j=a_{p+1}e_q+a_1e_{q+1}$ else; when $m=pn+1$, $x_m=a_{p+1}e_n$. We put $W_{p,n,m}=F/K$. 
 We can say that $W_{p,n,m}$ is named by the following $n\times m$ matrix, where $r=m-p(n-1)$:

\bigskip

\(\left( \begin{matrix}
a_1 &  .. & a_p & a_{p+1} & 0 & \hdotsfor[0.7]{4}\\
0 & .. & 0 & a_1 & a_2 & \dots &  a_{p+1} & 0  & .. \\
\vdots & & & & & & &  & \ddots\\
0 & \hdotsfor[0.7]{7} & \\
0 & \hdotsfor[0.7]{8} 
\end{matrix}
\begin{matrix}
 \hdotsfor[0.7]{7} & 0 \\
 \hdotsfor[0.7]{7}& 0 \\
& & & & & &  &  \vdots \\
 0 & a_1  & .. & a_p & a_{p+1} & 0 & .. & 0 \\
\hdotsfor[0.7]{3} & 0 & a_1 & a_2 & .. & a_r
\end{matrix}\right) 
\)

\bigskip

Since $K\subseteq PF$, $\mathrm{gen}\ W_{p,n,m}=n$. Now we consider the following relation: $\sum_{j=0}^{m}c_jx_j=0$. From the definition of the $x_j$ we get the following equality:
\[\sum_{q=0}^{n-2}\left( \sum_{i=1}^{p+1}c_{pq+i}a_i\right) e_{q+1}+\left( \sum_{i=1}^{r}c_{p(n-1)+i}a_i\right) e_n=0.\]
Since $\{e_1,\dots,e_n\}$ is a basis and $\mathrm{gen}\ A=p+1$ we deduce that $c_j\in P,\ \forall j,\ 1\leq j\leq m$. So, $\mathrm{rel}\ W_{p,n,m}=m$.

Let $s\in \mathrm{End}_R(W_{p,n,m})$. Then $s$ is
induced by an endomorphism $\tilde{s}$ of $F$ which satisfies
$\tilde{s}(K)\subseteq K$. For each $j,\ 1\leq j\leq n,$ there exists a family $(\alpha_{i,j})$ of elements of $R$ such that:
\begin{equation} \label{eq:s1}
\tilde{s}(e_j)=\sum_{i=1}^{n}\alpha_{i,j}e_i
\end{equation}
Since $\tilde{s}(K)\subseteq K,\ \forall j, 1\leq j\leq  m,\ \exists$ a
 family $(\beta_{i,j})$ of elements of $R$ such that:
\begin{equation} \label{eq:s2}
\tilde{s}(x_j)=\sum_{i=1}^{m}\beta_{i,j}x_i
\end{equation}
From (\ref{eq:s1}), (\ref{eq:s2}) and the equality $x_1=a_1e_1$
if follows that:

\[\sum_{q=1}^{n}\alpha_{q,1}a_1e_q=\sum_{q=0}^{n-2}\left( \sum_{i=1}^{p+1}\beta_{pq+i,1}a_i\right) e_{q+1}+\left( \sum_{i=1}^{r}\beta_{p(n-1)+i,1}a_{i}\right) e_n.\]

Then,  we get:
\[\forall q,\ 1\leq q\leq n-1,\qquad \alpha_{q,1}a_1=\sum_{i=1}^{p+1}\beta_{p(q-1)+i,1}a_i\]
\[\mathrm{and}\qquad\alpha_{n,1}a_1=\sum_{i=1}^{r}\beta_{p(n-1)+i,1}a_{i}.\]
We deduce that: $\forall q,\ 2\leq q\leq n,\ \beta_{p(q-2)+p+1,1}\in P$ and $\beta_{p(q-1)+1,1}\equiv \alpha_{q,1}\ [P]$. So, 
\begin{equation}\label{eq:s3}
\forall q,\ 2\leq q\leq n,\ \alpha_{q,1}\in P.
\end{equation}

Now, let $j=p\ell+1$ where $1\leq \ell\leq (n-1)$. In this case, $x_j=a_{p+1}e_{\ell}+a_1e_{\ell+1}$. From (\ref{eq:s1}) and (\ref{eq:s2}) it follows that:
\[\sum_{q=1}^{n}(\alpha_{q,\ell}a_{p+1}+\alpha_{q,\ell+1}a_1)e_q=\sum_{q=0}^{n-2}\left( \sum_{i=1}^{p+1}\beta_{pq+i,j}a_i\right) e_{q+1}+\left( \sum_{i=1}^{r}\beta_{p(n-1)+i,j}a_{i}\right) e_n.\]
Then,  we get:
\[\forall q,\ 1\leq q\leq n-1,\qquad \alpha_{q,\ell}a_{p+1}+\alpha_{q,\ell+1}a_1=\sum_{i=1}^{p+1}\beta_{p(q-1)+i,j}a_i\]
\[\mathrm{and}\qquad\alpha_{n,\ell}a_{p+1}+\alpha_{n,\ell+1}a_1=\sum_{i=1}^{r}\beta_{p(n-1)+i,j}a_{i}.\]
We deduce that \[\forall q,\ell,\ 1\leq q,\ell\leq (n-1),\ \alpha_{q,\ell}\equiv \beta_{p(q-1)+p+1,j}\ [P]\ \mathrm{and}\ \alpha_{q+1,\ell+1}\equiv \beta_{pq+1,j}\ [P],\] whence $\alpha_{q,\ell}\equiv\alpha_{q+1,\ell+1}\ [P]$. Consequently, $\forall q,\ 1\leq q\leq n,\ \alpha_{q,q}\equiv \alpha_{1,1}\ [P]$ and $\forall t,\ 1\leq t\leq (n-1),\ \forall q,\ 1\leq q\leq (n-t),\ \alpha_{q+t,q}\equiv \alpha_{1+t,1}\equiv 0\ [P]$ by (\ref{eq:s3}). Let $\bar{s}$ be the endomorphism of $W_{p,n,m}/PW_{p,n,m}$ induced by $s$. If $\alpha_{1,1}$ is a unit then $\bar{s}$ is an isomorphism, else $\overline{\mathbf{1}_{W_{p,n,m}}-s}$ is an isomorphism. By Lemma~\ref{L:iso} we conclude that either $s$  or $(\mathbf{1}_{W_{p,n,m}}-s)$ is an isomorphism.  Hence, $\mathrm{End}_R(W_{p,n,m})$ is local. \end{proof}

\begin{remark}
Observe that $\mathrm{D}(W_{1,n-1,n})$ is isomorphic to the indecomposable module built  in the proof of \cite[Theorem 2]{War70}.
\end{remark}

\begin{theorem}
 \label{T:ComPur} Let $R$ be a commutative ring. The following assertions hold:
\begin{enumerate}
\item Assume that, for any integer $p>0$, there exists a maximal ideal $P$ and a  finitely generated ideal $A$ of $R_P$ such that $\mathrm{gen}_{R_P}\ A\geq p+1$. Then, if $(n,m)$ and $(r,s)$ are two different pairs of integers, the $(n,m)$-purity and the $(r,s)$-purity are not equivalent.
\item Assume that, there exists an integer $p>0$ such that, for each maximal ideal $P$, for any finitely generated ideal $A$ of $R_P$, $\mathrm{gen}_{R_P}\ A\leq p$. Then:
\begin{enumerate}
\item for each integer $n>0$ the $(\aleph_0,n)$-purity (respectively $(n,\aleph_0)$-purity) is equivalent to the $(np,n)$-purity (respectively $(n,np)$-purity);
\item if $p>1$, then, for each integer $n>0$, for each integer $m,\ 1\leq m\leq n(p-1)$, the $(n,m)$-purity (respectively $(m,n)$-purity) is not equivalent to the $(n,m+1)$-purity (respectively $(m+1,n)$-purity).
\end{enumerate}
\end{enumerate}
\end{theorem}
\begin{proof} By Proposition~\ref{P:PureLocal} we may assume that $R$ is local with maximal $P$. By Theorem~\ref{T:pure} the $(n,m)$-purity and the $(r,s)$-purity are equivalent if and only if so are the $(m,n)$-purity and the $(s,r)$-purity.
 
(1). Suppose that $r>n$ and let $t=\min(m,s)$. Let $q$ be the greatest divisor of $(r-1)$ which is $\leq t$ and $p=(r-1)/q$. Let $A$ be a finitely generated ideal such that $\mathrm{gen}\ A>p$. By way of contradiction, suppose that $W_{p,q,r}$ is $(n,m)$-pure-projective. By Proposition~\ref{P:PureProj} $W_{p,q,r}$ is a summand of $\oplus_{i\in I}F_i$ where $I$ is a finite set and $\forall i\in I,\ F_i$ is a $(m,n)$-presented $R$-module. Since its  endomorphism ring is local, $W_{p,q,r}$ is an exchange module (see \cite[Theorem 2.8]{Fac98}). So, we have $W_{p,q,r}\oplus(\oplus_{i\in I}G_i)\cong(\oplus_{i\in I}H_i)\oplus(\oplus_{i\in I}G_i)$ where $\forall i\in I$, $G_i$ and $H_i$ are submodules of $F_i$ and $F_i=G_i\oplus H_i$. Let $G=\oplus_{i\in I}G_i$. Then $G$ is finitely generated. By \cite[Proposition V.7.1]{FuSa01} $\mathrm{End}_R(G)$ is semilocal. By using Evans's theorem (\cite[Corollary 4.6]{Fac98}) we deduce that $W_{p,q,r}\cong(\oplus_{i\in I}H_i)$. Since $W_{p,q,r}$ is indecomposable, we get that it is $(m,n)$-presented. This contradicts that $\mathrm{rel}\ W_{p,q,r}=r>n$.

(2)(a) is an immediate consequence of Proposition~\ref{P:genIdeal}.

(2)(b). There exist two integers $q,t$ such that $m+1=(q-1)(p-1)+t$ with $n\geq q\geq 1$ and $1\leq t\leq p$. As in (1) we prove that $W_{p-1,q,m+1}$ is not $(m,n)$-pure-projective. 
\end{proof}

\begin{remark}
In the previous theorem, when there exists an integer $p>1$ such that, for any finitely generated ideal $A$ $\mathrm{gen}\ A\leq p$, we don't know if the $(n,m)$-purity and the  $(n,m+1)$-purity are equivalent when $n(p-1)+1\leq m\leq np-1$. If $R$ is a local ring with maximal $P$ with residue field $k$ such that $P^2=0$ and $\dim_kP=p$ it is easy to show that each finitely presented $R$-module $F$ with $\mathrm{gen}\ F=n$ and $\mathrm{rel}\ F=np$ is semisimple. So, the $(np,n)$-purity is equivalent to the $(np-1,n)$-purity.
\end{remark}

\section{$(n,m)$-flat modules and $(n,m)$-injective modules}
\label{S:flat}
Let $M$ be a right $R$-module. We say that $M$ is $(n,m)${\it -flat} if for any $m$-generated submodule $K$ of a $n$-generated free left $R$-module $F$, the natural map: $M\otimes_RK\rightarrow M\otimes_RF$ is a monomorphism. We say that $M$ is $(\aleph_0,m)${\it -flat} (respectively $(n,\aleph_0)${\it -flat}) if $M$ is $(n,m)$-flat for each integer $n>0$ (respectively $m>0$). We say that $M$ is $(n,m)${\it -injective} if for any $m$-generated submodule $K$ of a $n$-generated free right $R$-module $F$, the natural map: $\mathrm{Hom}_R(F,M)\rightarrow \mathrm{Hom}_R(K,M)$ is an epimorphism. We say that $M$ is $(\aleph_0,m)${\it -injective} (respectively $(n,\aleph_0)${\it -injective}) if $M$ is $(n,m)$-injective for each integer $n>0$ (respectively $m>0$). A ring $R$ is called {\it left self $(n,m)$-injective} if $R$ is $(n,m)$-injective as left $R$-module.

If $R$ is a commutative domain, then an $R$-module is $(1,1)$-flat (respectively $(1,1)$-injective) if and only if it is torsion-free (respectively divisible).

The following propositions can be proved with standard technique: see \cite[Theorem 4.3 and Proposition 2.3]{ZhChZh05}. In these propositions the integers $n$ or $m$ can be replaced with $\aleph_0$.
\begin{proposition}
\label{P:flat} Assume that $R$ is an algebra over a commutative ring $S$ and let $E$ be an injective $S$-cogenerator. Let $M$ be a right $R$-module. The following conditions are equivalent:
\begin{enumerate}
\item $M$ is $(n,m)$-flat;
\item each exact sequence $0\rightarrow L\rightarrow N \rightarrow M\rightarrow 0$ is $(n,m)$-pure, where $L$ and $N$ are right $R$-modules;
\item for each $(m,n)$-presented right module $F$, every homomorphism $f:F\rightarrow M$ factors through a free right $R$-module;
\item $\mathrm{Hom}_S(M,E)$ is a $(n,m)$-injective left $R$-module.
\end{enumerate}
\end{proposition}

\begin{proposition}
\label{P:inj} Let $M$ be a right module. The following conditions are equivalent:
\begin{enumerate}
\item $M$ is $(n,m)$-injective;
\item each exact sequence $0\rightarrow M\rightarrow L\rightarrow N\rightarrow 0$ is $(m,n)$-pure, where $L$ and $N$ are right $R$-modules;
\item $M$ is a $(m,n)$-pure  submodule of its injective hull.
\end{enumerate}
\end{proposition}

\begin{proposition}
\label{P:localflat} Let $R$ be a commutative ring. Then an $R$-module $M$ is $(n,m)$-flat if and only if, for each maximal ideal $P$, $M_P$ is $(n,m)$-flat over $R_P$.
\end{proposition}

\begin{lemma}
\label{L:p-gen} Let $M$ be $p$-generated right $R$-module where $p$ is a positive integer. Then $M$ is flat if and only if it is $(1,p)$-flat.
\end{lemma}
\begin{proof}
Only ``if'' requires a proof. Let $A$ be a left ideal. Assume that $M$ is generated by $x_1,\dots,x_p$. So, if $z\in M\otimes_RA$, $z=\sum_{i=1}^px_i\otimes a_i$ where $a_1,\dots,a_p\in A$. Suppose that the image of $z$ in $M\otimes_RR$ is $0$. If $A'$ is the left ideal generated by $a_1,\dots,a_p$, if $z'$ is the element of $M\otimes_RA'$ defined by $z'=\sum_{i=1}^px_i\otimes a_i$, then $z$ (respectively $0$) is the image of $z'$ in $M\otimes_RA$ (respectively $M\otimes_RR$). Since $M$ is $(1,p)$-flat we successively deduce that $z'=0$ and $z=0$.
\end{proof}

\bigskip
It is well known that each $(1,\aleph_0)$-flat right module is $(\aleph_0,\aleph_0)$-flat. {\bf For each positive integer $p$, is each $(1,p)$-flat right module $(\aleph_0,p)$-flat?} 

The following theorem and Theorem~\ref{T:parfait} give a partial answer to this question.

\begin{theorem}
\label{T:p-flatideal} Let $p$ be a positive integer and let $R$ be a ring. For each positive integer $n$, assume that, for each $p$-generated submodule $G$ of the left $R$-module $R^n\oplus R$, $(G\cap R^n)$ is the direct limit of its $p$-generated submodules.Then a right $R$-module $M$ is $(1,p)$-flat if and only if it is $(\aleph_0,p)$-flat.
\end{theorem}
\begin{proof}
We shall prove that $M$ is $(n,p)$-flat by induction on $n$. Let $G$ be a $p$-generated submodule of the left $R$-module $R^{n+1}=R^n\oplus R$. Let $\pi$ be the projection of $R^{n+1}$ onto $R$ and $G'=\pi(G)$. Then $G'$ is a $p$-generated left module. We put $H=G\cap R^n$.
We have the following commutative diagram with exact horizontal sequences:
\[\begin{matrix}
{} & {} & M\otimes_RH &\rightarrow & M\otimes_RG &\xrightarrow{1_M\otimes\pi} & M\otimes_RG'& \rightarrow & 0\\
{} & {} & \downarrow & {} & \downarrow & {} & \downarrow & {} & {} \\
0 & \rightarrow & M\otimes_R R^n &\rightarrow & M\otimes_RR^{n+1} &\xrightarrow{1_M\otimes\pi} & M\otimes_RR& \rightarrow & 0
\end{matrix}\]
Let $u:G\rightarrow R^{n+1},\ u':G'\rightarrow R,\ w:R^n\rightarrow R^{n+1}$ be the inclusion maps and let $v=u\vert_{H}$. Then $(1_M\otimes u')$ is injective. Let $H'$ be a $p$-generated submodule of $H$. By the induction hypothesis  $M$ is $(n,p)$-flat. So, $(1_M\otimes(v\vert_{H'}))$ is injective. It follows that $(1_M\otimes v)$ is injective too. We conclude that $(1_M\otimes u)$ is injective and $M$ is $(\aleph_0,p)$-flat.
\end{proof}

\begin{corollary}
\label{C:p-flatideal} Let $p$ be a positive integer and let $R$ be a ring such that each left ideal is $(1,p)$-flat. Then, for each positive integer $q\leq p$, a right $R$-module $M$ is $(1,q)$-flat if and only if it is $(\aleph_0,q)$-flat.
\end{corollary}
\begin{proof}
Let the notations be as in the previous theorem. Since $G'$ is a flat left $R$-module by Lemma~\ref{L:p-gen}, $H$ is a pure submodule of $G$. Let $\{g_1,\dots,g_q\}$ be a spanning set of $G$ and let $h_1,\dots,h_t\in H$. For each $k$, $1\leq k\leq t$ there exist $a_{k,1},\dots,a_{k,q}\in R$ such that $h_k=\sum_{i=1}^qa_{k,i}g_i$. It follows that there exist $g'_1,\dots,g'_q\in H$ such that $\forall k,\ 1\leq k\leq t,\ h_k=\sum_{i=1}^qa_{k,i}g'_i$. So, each finitely generated submodule of $H$ is contained in a $q$-generated submodule. We conclude by applying Theorem~\ref{T:p-flatideal}.
\end{proof}

\begin{corollary} 
\label{C:OneAlepFlat} Let $R$ be a commutative local ring with maximal $P$. Assume that $P^2=0$. Let $q$ a positive integer. Then:
\begin{enumerate}
\item  each $(1,q)$-flat module is $(\aleph_0,q)$-flat;
\item  each $(1,q)$-injective module is $(\aleph_0,q)$-injective.
\end{enumerate}
\end{corollary}
\begin{proof}  Let the notations be as in the previous theorem. We may assume that $G\subseteq PR^{n+1}$. Then $G$ is a semisimple module and $H$ is a direct summand of $G$. So, $(1)$ is a consequence of Theorem~\ref{T:p-flatideal}.

$(2)$. Let $M$ be a $(1,q)$-injective module. We shall prove by induction on $n$ that $M$ is $(n,q)$-injective. We have the following commutative diagram:
\[\begin{matrix}
0 & \rightarrow & \mathrm{Hom}_R(R,M)& \rightarrow & \mathrm{Hom}_R(R^{n+1},M)& \rightarrow & \mathrm{Hom}_R(R^n,M)& \rightarrow & 0\\
{} & {} & \downarrow & {} & \downarrow & {} & \downarrow & {} & {} \\
0 & \rightarrow & \mathrm{Hom}_R(G',M)& \rightarrow & \mathrm{Hom}_R(G,M)& \rightarrow & \mathrm{Hom}_R(H,M)& \rightarrow & 0 
\end{matrix}\]
where the horizontal sequences are exact. By the induction hypothesis the left and the right vertical maps are surjective. It follows that the middle vertical map is surjective too.
\end{proof}

\bigskip
By \cite[Example 5.2]{Sha01} or \cite[Theorem 2.3]{Jon71}, for each integer $n>0$, there exists a ring $R$ for which each finitely generated left ideal is $(1,n)$-flat (hence $(\aleph_0,n)$-flat by Corollary~\ref{C:p-flatideal}) but there is a finitely generated left ideal which is not $(1,n+1)$-flat. The following proposition gives other examples in the commutative case.

\begin{proposition}
\label{P:ExFlat} Let $R$ be a commutative local ring with maximal ideal $P$ and residue field $k$. Assume that $P^2=0$ and $\dim_k\ P>1$. Then, for each positive integer $p<\dim_k\ P$, there exists:
\begin{enumerate}
\item a $(p+1,1)$-presented  $R$-module which is $(\aleph_0,p)$-flat but not $(1,p+1)$-flat;
\item a $(\aleph_0,p)$-injective $R$-module which is not $(1,p+1)$-injective.
\end{enumerate}
\end{proposition}
\begin{proof}
(1). Let $F$ be a free $R$-module of rank $(p+1)$ with basis $\{e_1,\dots,e_p,e_{p+1}\}$, let $(a_1,\dots,a_p,a_{p+1})$ be a family of linearly independent elements of $P$, let $K$ be the submodule of $F$ generated by $\sum_{i=1}^{p+1}a_ie_i$ and let $M=F/K$. Then $M\cong D(W_{p,1,p+1})$ (see the  proof of Proposition~\ref{P:locEnd}). First, we show that $K$ is a $(1,p)$-pure submodule of $F$. We consider the following equation:
\begin{equation}\label{eq:s4}
\sum_{j=1}^pr_jx_j=s(\sum_{i=1}^{p+1}a_ie_i)
\end{equation}
where $r_1,\dots,r_p,s\in R$ and with unknowns $x_1,\dots,x_p$. Assume that this equation has a solution in $F$. Suppose there exists $\ell,\ 1\leq\ell\leq p$, such that $r_{\ell}$ is a unit. For each $j,\ 1\leq j\leq p$, we put $x'_j=\delta_{j,\ell}r_{\ell}^{-1}s(\sum_{i=1}^{p+1}a_ie_i)$. It is easy to check that $(x'_1,\dots,x'_p)$ is a solution of (\ref{eq:s4}) in $K$. Now we assume that $\ r_j\in P,\ \forall j,\ 1\leq j\leq p$. Suppose that $(x_1,\dots,x_p)$ is a solution of (\ref{eq:s4}) in $F$. For each $j,\ 1\leq j\leq p$, $x_j=\sum_{i=1}^{p+1}c_{j,i}e_i$, where $c_{j,i}\in R$. We get the following equality:
\begin{equation}
\sum_{i=1}^{p+1}\left( \sum_{j=1}^pr_jc_{j,i}\right) e_i=\sum_{i=1}^{p+1}sa_ie_i
\end{equation}
We deduce that:
\begin{equation}
\forall i,\quad 1\leq i\leq p+1,\quad\sum_{j=1}^pr_jc_{j,i}=sa_i
\end{equation}
So, if $s$ is a unit, $\forall i,\ 1\leq i\leq p+1,\ a_i\in\sum_{j=1}^pRr_j$. It  follows that \[\dim_k\left( \sum_{i=1}^{p+1}Ra_i\right) \leq p\] that is false. So, $s\in P$. In this case (\ref{eq:s4}) has the nil solution. Hence $M$ is $(\aleph_0,p)$-flat by Proposition~\ref{P:flat}(2) and Corollary~\ref{C:OneAlepFlat}.

By way of contradiction suppose that $M$ is $(1,p+1)$-flat. It follows that $K$ is a $(1,p+1)$-pure submodule of $F$ by Proposition~\ref{P:flat}. Since $M$ is $(1,p+1)$-pure-projective we deduce that $M$ is free. This is false.

(2). Let $E$ be an injective $R$-cogenerator. Then $\mathrm{Hom}_R(M,E)$ is $(\aleph_0,p)$-injective but not $(1,p+1)$-injective by Proposition~\ref{P:flat}(4).
\end{proof}

\medskip
In a similar way we show the following proposition.

\begin{proposition}
Let $R$ be a commutative local ring with maximal ideal $P$. Assume that $P^2=0$. Let $M$ be  a $(m,1)$-presented $R$-module with $m>1$, let $\{x_1,\dots,x_m\}$ be a spanning set of $M$ and let $\sum_{j=1}^ma_jx_j=0$ be the relation of $M$, where $a_1,\dots,a_m\in P$.  If $p=\mathrm{gen}\ (\sum_{j=1}^mRa_j)-1>0$, then:
\begin{enumerate}
\item $M$ is $(\aleph_0,p)$-flat but  not $(1,m)$-flat;
\item $\mathrm{Hom}_R(M,E)$ is $(\aleph_0,p)$-injective but not $(1,m)$-injective, where $E$ is an injective $R$-cogenerator.
\end{enumerate}
\end{proposition}

When $R$ is an arithmetical commutative ring, i.e. its lattice of ideals is distributive, each $(1,1)$-flat module is flat and by \cite[Theorem VI.9.10]{FuSa01} the converse holds if $R$ is a commutative domain (it is also true if each principal ideal is flat).  However we shall see that there exist non-arithmetical commutative rings for which each $(1,1)$-flat module is flat. Recall that a left (or right) $R$-module $M$ is {\it torsionless} if the natural map $M\rightarrow (M^*)^*$ is injective.

\begin{proposition}
\label{P:AlephInj} For each ring $R$ the following conditions are equivalent:
\begin{enumerate}
\item $R$ is right self $(\aleph_0,1)$-injective;
\item each finitely presented cyclic left $R$-module is torsionless;
\item each finitely generated left ideal $A$ satisfies $A=\mathrm{l-ann}(\mathrm{r-ann}(A)).$
\end{enumerate}
\end{proposition}
\begin{proof}
We prove $(1)\Leftrightarrow (2)$ as \cite[Theorem 2.3]{Ja73} and $(2)\Leftrightarrow (3)$ is easy.
\end{proof}

\begin{theorem}
\label{T:parfait} Let $R$ be a right perfect ring which is right self $(\aleph_0,1)$-injective. Then each $(1,1)$-flat right module is projective.
\end{theorem}
\begin{proof}
Let $M$ be a $(1,1)$-flat right $R$-module. It is enough to show that $M$ is flat. Let $A$ be a finitely generated left ideal of $R$. Assume that $\{a_1,\dots,a_n\}$ is a minimal system of generators of $A$ with $n>1$. Let $z\in M\otimes_RA$ such that its image in $M$ is $0$. We have $z=\sum_{i=1}^ny_i\otimes a_i$, where $y_1,\dots,y_n\in M$, and $\sum_{i=1}^ny_ia_i=0$. For each $i,\ 1\leq i\leq n$, we set $A_i=\sum_{\binom{j=1}{j\ne i}}^nRa_j$. Then, $\forall i,\ 1\leq i\leq n$, $A_i\subset A$. For each finitely generated left ideal $B$ we have $B=\mathrm{l-ann}(\mathrm{r-ann}(B))$. It follows that, $\forall i,\ 1\leq i\leq n$, $\mathrm{r-ann}(A)\subset \mathrm{r-ann}(A_i)$. Let $b_i\in \mathrm{r-ann}(A_i))\setminus \mathrm{r-ann}(A)$. Then $y_ia_ib_i=0$. From the $(1,1)$-flatness of $M$ we deduce that $y_i=\sum_{k=1}^{m_i}y'_{i,k}c_{i,k}$, where $y'_{i,1},\dots,y'_{i,m_i}\in M$ and $c_{i,1},\dots,c_{i,m_i}\in R$ with $c_{i,k}a_ib_i=0,\ \forall k,\ 1\leq k\leq m_i$. It follows that $z=\sum_{i=1}^n(\sum_{k=1}^{m_i}y'_{i,k}\otimes c_{i,k}a_i)$. Let $A^{(1]}$ be the left ideal generated by $\{c_{i,k}a_i\mid 1\leq i\leq n,\ 1\leq k\leq m_i\}$. Then $A^{(1)}\subset A$; else, $\forall i,\ 1\leq i\leq n$, $a_i=\sum_{j=1}^{n}(\sum_{k=1}^{m_j}d_{i,j,k}c_{j,k}a_j)$ with $d_{i,j,k}\in R$; we get that $a_ib_i=\sum_{j=1}^{n}(\sum_{k=1}^{m_j}d_{i,j,k}c_{j,k}a_jb_i)$; but $a_jb_i=0$ if $j\ne i$ and $c_{i,k}a_ib_i=0$; so, there is a contradiction because the second member of the previous equality is $0$ while $a_ib_i\ne 0$ . Let $\{a^{(1)}_1,\dots,a^{(1)}_{n_1}\}$ be a minimal system of generators of $A^{(1)}$. So, $z=\sum_{i=1}^{n_1}y^{(1)}_i\otimes a^{(1)}_i$ where $y^{(1)}_1,\dots,y^{(1)}_{n_1}\in M$, and $z$ is the image of $z^{(1)}\in M\otimes_RA^{(1)}$ defined by $z^{(1)}=\sum_{i=1}^{n_1}y^{(1)}_i\otimes a^{(1)}_i$. If $n_1\leq 1$ we conclude that $z^{(1)}=0$ since $M$ is $(1,1)$-flat, and $z=0$. If $n_1>1$, in the same way we get that $z^{(1)}$ is the image of an element $z^{(2)}\in M\otimes_RA^{(2)}$ where $A^{(2)}$ is a left ideal such that $A^{(2)}\subset A^{(1)}$. If $\mathrm{gen}\ A^{(2)}>1$ we repeat this process, possibly several times, until we get a left ideal $A^{(l)}$ with $\mathrm{gen}\ A^{(l)}\leq 1$; this is possible because $R$ satisfies the descending chain condition on finitely generated left ideals since it is right perfect (see \cite[Th\'eor\`eme 5 p.130]{Ren75}). The $(1,1)$-flatness of $M$ implies that $z^{(l)}=0$ and $z=0$. So, $M$ is projective.
\end{proof}

Let $\mathcal{P}$ be a ring property. We say that a commutative ring $R$ is {\it locally $\mathcal{P}$} if $R_P$ satisfies $\mathcal{P}$ for each maximal ideal $P$. 

The following corollary is a consequence of Theorem~\ref{T:parfait} and Proposition~\ref{P:localflat}.
\begin{corollary}
\label{C:LocPer} Let $R$ be a commutative ring which is locally perfect and locally self $(\aleph_0,1)$-injective. Then each $(1,1)$-flat $R$-module is flat.
\end{corollary}




\section{$(n,m)$-coherent rings}
\label{S:coherent}
We say  that a ring $R$ is left $(n,m)${\it -coherent} if each $m$-generated submodule of a $n$-generated free left $R$-module is finitely presented. We say that $R$ is left $(\aleph_0,m)${\it -coherent} (respectively $(n,\aleph_0)${\it -coherent}) if for each integer $n>0$ (respectively $m>0$) $R$ is left $(n,m)$-coherent. The following theorem can be proven with standard technique: see  \cite[Theorems 5.1 and 5.7]{ZhChZh05}. In this theorem the integers $n$ or $m$ can be replaced with $\aleph_0$.
\begin{theorem}\label{T:coh}
Let $R$ be a ring and $n, m$ two fixed positive integers. Assume that $R$ is an algebra over a commutative ring $S$. Let $E$ be an injective $S$-cogenerator. Then the following conditions are equivalent:
\begin{enumerate}
\item $R$ is left $(n,m)$-coherent;
\item any direct product of right $(n,m)$-flat $R$-modules is $(n,m)$-flat;
\item for any set $\Lambda$, $R^{\Lambda}$ is a $(n,m)$-flat right $R$-module;
\item any direct limit of a direct system of $(n,m)$-injective left $R$-modules is $(n,m)$-injective;
\item for any exact sequence of left modules $0\rightarrow A\rightarrow B\rightarrow C\rightarrow 0$, $C$ is $(n,m)$-injective if so is $B$ and if $A$ is a $(\aleph_0,m)$-pure submodule of $B$;
\item for each $(n,m)$-injective left $R$-module $M$, $\mathrm{Hom}_S(M,E)$ is $(n,m)$-flat.
\end{enumerate}
\end{theorem}

\bigskip
It is well known that each left $(1,\aleph_0)$-coherent ring is left $(\aleph_0,\aleph_0)$-coherent. {\bf For each positive integer $p$, is each left $(1,p)$-coherent ring left $(\aleph_0,p)$-coherent?} 

Propositions~\ref{P:PF=PP} and \ref{P:CommParf} and Theorem~\ref{T:parfait2} give a partial answer to this question.

\begin{proposition}
\label{P:PF=PP} Let $p$ be a positive integer and let $R$ be a ring. For each positive integer $n$, assume that, for each $p$-generated submodule $G$ of the left $R$-module $R^n\oplus R$, $(G\cap R^n)$ is the direct limit of its $p$-generated submodules. Then the following conditions are equivalent:
\begin{enumerate}
\item $R$ is left $(1,p)$-coherent;
\item $R$ is left $(\aleph_0,p)$-coherent.
\end{enumerate}
Moreover, when these conditions hold each $(1,p)$-injective left module is $(\aleph_0,p)$-injective.
\end{proposition}
\begin{proof}
It is obvious that $(2)\Rightarrow (1)$. 

$(1)\Rightarrow (2)$. Let $\Lambda$ be a set. By Theorem~\ref{T:coh} $R^{\Lambda}$ is a $(1,p)$-flat right module. From Theorem~\ref{T:p-flatideal} we deduce that $R^{\Lambda}$ is a $(\aleph_0,p)$-flat right module. By using again Theorem~\ref{T:coh} we get $(2)$. 

Let $M$ be a $(1,p)$-injective left module. By Theorem~\ref{T:coh} $M^{\sharp}$ is a $(1,p)$-flat right $R$-module. Then it is also $(\aleph_0,p)$-flat. We deduce that $(M^{\sharp})^{\sharp}$ is a $(\aleph_0,p)$-injective left module. Since $M$ is a pure submodule of $(M^{\sharp})^{\sharp}$, it follows that $M$ is $(\aleph_0,p)$-injective too.
\end{proof}
\begin{proposition}
\label{P:CommParf} Let $R$ be a commutative perfect ring. Then $R$ is Artinian if and only if it is $(1,1)$-coherent.
\end{proposition}
\begin{proof}
Suppose that $R$ is $(1,1)$-coherent. Since $R$ is perfect, $R$ is a finite product of local rings. So, we may assume that $R$ is local with maximal $P$. Let $S$ be a minimal non-zero ideal of $R$ generated by $s$. Then $P$ is the annihilator of $s$. So, $P$ is finitely generated and it is the sole prime ideal of $R$. Since all prime ideals of $R$ are finitely generated, $R$ is Noetherian. On the other hand $R$ satisfies the descending chain condition on finitely generated ideals. We conclude that $R$ is Artinian.
\end{proof}

\bigskip
Except in some particular cases, we don't know if each $(1,p)$-injective module is $(\aleph_0,p)$-injective, even if we replace $p$ by $\aleph_0$.
\begin{theorem}
\label{T:parfait2} Let $R$ be a ring which is right perfect, left $(1,1)$-coherent and right self $(\aleph_0,1)$-injective. Then each $(1,1)$-injective left module is $(\aleph_0,\aleph_0)$-injective and $R$ is left coherent.
\end{theorem}
\begin{proof}
Let $M$ be a left $(1,1)$-injective module. By Theorem~\ref{T:coh} $M^{\sharp}$ is $(1,1)$-flat. Whence $M^{\sharp}$ is projective by Theorem~\ref{T:parfait}. We do as in the proof of  Proposition~\ref{P:PF=PP} to conclude that $M$ is $(\aleph_0,\aleph_0)$-injective.

For each set $\Lambda$, $R^{\Lambda}$ is a $(1,1)$-flat right module by Theorem~\ref{T:coh}. It follows that $R^{\Lambda}$ is a projective right module by Theorem~\ref{T:parfait}.
\end{proof}

Recall that a ring is {\it quasi-Frobenius} if it is Artinian and self-injective.
\begin{corollary}
\label{C:Frob} Let $R$ be a quasi-Frobenius ring. Then, for each right (or left) $R$-module $M$, the following conditions are equivalent:
\begin{enumerate}
\item $M$ is $(1,1)$-flat;
\item $M$ is projective;
\item $M$ is injective;
\item $M$ is $(1,1)$-injective.
\end{enumerate}
\end{corollary}
\begin{proof}
It is well known that $(2)\Leftrightarrow (3)$. By Theorem~\ref{T:parfait} $(1)\Leftrightarrow (2)$ because $R$ satisfies the conditions of this theorem, and it is obvious that $(3)\Rightarrow (4)$ and the converse holds by Theorem~\ref{T:parfait2}.
\end{proof}

We prove the following theorem as \cite[Th\'eor\`eme 1.4]{Cou82}.

\begin{theorem}\label{T:locCoh}
\label{T:pure-projective} Let $R$ be a commutative ring and $n, m$ two fixed positive integers. The following conditions are equivalent:
\begin{enumerate}
\item $R$ is $(n,m)$-coherent;
\item for each multiplicative subset $S$ of $R$, $S^{-1}R$ is $(n,m)$-coherent, and for each $(n,m)$-injective $R$-module $M$, $S^{-1}M$ is $(n,m)$-injective over $S^{-1}R$;
\item For each maximal ideal $P$, $R_P$ is $(n,m)$-coherent and for each $(n,m)$-injective $R$-module $M$, $M_P$ is $(n,m)$-injective over $R_P$.
\end{enumerate} 
\end{theorem}

Recall that a ring $R$ is a {\it right IF-ring} if each right injective $R$-module is flat.
\begin{theorem}
\label{T:LocPerIF} Let $R$ be a commutative ring which is locally perfect, $(1,1)$-coherent and self $(1,1)$-injective. Then:
\begin{enumerate}
\item $R$ is coherent, self $(\aleph_0,\aleph_0)$-injective and  locally quasi-Frobenius;
\item each $(1,1)$-flat module is flat;
\item each $(1,1)$-injective module is $(\aleph_0,\aleph_0)$-injective.
\end{enumerate}
\end{theorem}
\begin{proof}
By Theorem~\ref{T:locCoh} $R_P$ is $(1,1)$-coherent and $(1,1)$-injective for each maximal ideal $P$. Let $a$ be a generator of a minimal non-zero ideal of $R_P$. Then $PR_P$ is the annihilator of $a$ and consequently $PR_P$ is finitely generated over $R_P$. Since all prime ideals of $R_P$ are finitely generated, we deduce that $R_P$ is Artinian for each maximal ideal $P$. Moreover, the $(1,1)$-injectivity of $R_P$ implies that the socle of $R_P$ (the sum of all minimal non-zero ideals) is simple. It follows that $R_P$ is quasi-Frobenius for each maximal ideal $P$.

Let $M$ be a $(\aleph_0,\aleph_0)$-injective $R$-module. By Theorem~\ref{T:locCoh} $M_P$ is  $(1,1)$-injective for each maximal ideal $P$. By Corollary~\ref{C:Frob} $M_P$ is injective for each maximal ideal $P$. We conclude that $R$ is self $(\aleph_0,\aleph_0)$-injective and it is coherent by Theorem~\ref{T:locCoh}.

If $M$ is $(1,1)$-injective, we prove as above that $M_P$ is injective for each maximal ideal $P$. It follows that $M$ is $(\aleph_0,\aleph_0)$-injective.

The second assertion is an immediate consequence of Corollary~\ref{C:LocPer}.
\end{proof}

The following proposition is easy to prove:
\begin{proposition}
\label{P:1coh} A ring $R$ is left $(\aleph_0,1)$-coherent if and only if each finitely generated right ideal has a  finitely generated left annihilator.
\end{proposition}

\begin{example}
Let $V$ be a non-Noetherian (commutative) valuation domain whose  order group is not the additive group of real numbers and let $R=V[[X]]$ be the power series ring in one indeterminate over $V$. Since $R$ is a domain, $R$ is $(\aleph_0,1)$-coherent. But, in \cite{AnWa87} it is proven that there exist two elements $f$ and $g$ of $R$ such that $Rf\cap Rg$ is not finitely generated. By using the exact sequence $0\rightarrow Rf\cap Rg\rightarrow Rf\oplus Rg\rightarrow Rf+Rg\rightarrow 0$ we get that $Rf+Rg$ is not finitely presented. So, $R$ is not $(1,2)$-coherent.
\end{example}

\end{document}